\documentclass[a4paper,11pt]{amsart}

\usepackage{nccmath} 
\usepackage{amssymb,amsthm,mathtools}

\usepackage{mathrsfs}

\usepackage{mathabx}

\usepackage{filecontents}

\usepackage{hyperref}

\usepackage{enumitem}
\usepackage[all]{xy}
\xyoption{line}
\usepackage{color}
\usepackage[normalem]{ulem}

\renewcommand{\le}{\varleq}
\renewcommand{\ge}{\vargeq}

\newcommand\mL{L\kern-0.08cm\char39}

\newcommand{\myforall}{\text{ for all }}

\newcommand{\mywith}{\text{ with }}

\newcommand{\myand}{\text{ and }}
\newcommand{\myif}{\text{ if }}

\newcommand{\mythen}{\text{ then }}

\newcommand{\seb}{\{\,}
\newcommand{\sen}{\,\}}

\newcommand{\getsby}[1]{\xleftarrow{#1}}

\newcommand{\tesgh}{edge-surjective graph homomorphism}

\newcommand{\pdirectional}{\raise0.05em\hbox{$+$}directional}
\newcommand{\pdirectionality}{\raise0.05em\hbox{$+$}directionality}
\newcommand{\pdirectionalitys}{\raise0.05em\hbox{$+$}directionality }
\newcommand{\pdirectionals}{\raise0.05em\hbox{$+$}directional }
\newcommand{\mdirectional}{\raise0.05em\hbox{$-$}directional}
\newcommand{\mdirectionality}{\raise0.05em\hbox{$-$}directionality}
\newcommand{\mdirectionalitys}{\raise0.05em\hbox{$-$}directionality }
\newcommand{\mdirectionals}{\raise0.05em\hbox{$-$}directional }
\newcommand{\bidirectional}{bidirectional}

\newcommand{\Z}{\mathbb{Z}}
\newcommand{\Nonne}{\mathbb{N}}
\newcommand{\Posint}{\mathbb{N}^+}

\newcommand{\bi}{\in \Z}

\newcommand{\beposint}{\in \Posint}
\newcommand{\bpi}{\ge 1} 

\newcommand{\benonne}{\in \Nonne}
\newcommand{\bni}{\ge 0} 

\newcommand{\Gcal}{\mathcal{G}}

\newcommand{\sC}{\mathscr{C}}

\newcommand{\kuu}{\emptyset}
\newcommand{\nekuu}{\neq \kuu}

\newcommand{\fai}{\varphi}

\newcommand{\barx}{\bar{x}}
\newcommand{\bary}{\bar{y}}

\newcommand{\barC}{\bar{C}}

\newcommand{\barX}{\bar{X}}

\newcommand{\hx}{\hat{x}}
\newcommand{\hy}{\hat{y}}

\newcommand{\hX}{\hat{X}}

\newcommand{\ddC}{\ddot{C}}

\newcommand{\ddX}{\ddot{X}}

\newcommand{\centb}{\begin{center}}
\newcommand{\centn}{\end{center}}

\newcommand{\enumb}{\begin{enumerate}}
\newcommand{\enumn}{\end{enumerate}}

\newcommand{\itemb}{\begin{itemize}}
\newcommand{\itemn}{\end{itemize}}

\numberwithin{equation}{section}

\usepackage{cleveref}
\setlist[enumerate,1]{label=(\alph*),ref=(\alph*)}
\setlist[enumerate,2]{label=(\arabic*),ref=(\alph{enumi}-\arabic{enumii})}
\setlist[enumerate,3]{label=(\Alph*),ref=(\roman{enumi}-\alph{enumii}-\Alph*)}
\setlist[enumerate,4]{label=(\arabic*),ref=(\roman{enumi}-\alph{enumii}-\Alph{enumiii}-\arabic*)}

\newlist{alphaenum}{enumerate}{1}
\setlist[alphaenum,1]{label=($\alpha$-\arabic*),ref=($\alpha$-\arabic*)}
\newlist{betaenum}{enumerate}{1}
\setlist[betaenum,1]{label=($\beta$-\arabic*),ref=($\beta$-\arabic*)}
\newlist{gammaenum}{enumerate}{1}
\setlist[gammaenum,1]{label=($\gamma$-\arabic*),ref=($\gamma$-\arabic*)}

\newtheorem{thm}{Theorem}[section]

\newtheorem{prop}[thm]{Proposition}
\newtheorem{cor}[thm]{Corollary}

\theoremstyle{definition}
\newtheorem{defn}[thm]{Definition}

\theoremstyle{remark}

\newtheorem{nota}[thm]{Notation}
\newtheorem{rem}[thm]{Remark}

\crefname{sec}{\S}{\S\S}
\crefname{mainthm}{Theorem}{Theorems}
\crefname{thm}{Theorem}{Theorems}
\crefname{lem}{Lemma}{Lemmas}
\crefname{prop}{Proposition}{Propositions}
\crefname{cor}{Corollary}{Corollaries}
\crefname{defn}{Definition}{Definitions}
\crefname{conj}{Conjecture}{Conjectures}
\crefname{example}{Example}{Examples}
\crefname{nota}{Notation}{Notations}
\crefname{rem}{Remark}{Remarks}
\crefname{note}{Note}{Notes}
\crefname{case}{Case}{Cases}
\crefname{figure}{Figure}{Figures}
\crefname{section}{\S}{\S\S}
\crefname{enumi}{}{}
\crefname{enumii}{}{}
\crefname{equation}{}{}

\newcommand{\Vp}{V \setminus V_0}

\newcommand{\covrepa}[2]{#1_0 \getsby{#2_1} #1_1 \getsby{#2_2} #1_2 \getsby{#2_3} \dotsb}

\usepackage{url}

\begin{document}

\title[Topological rank does not increase by natural extension]
{Topological rank does not increase by natural extension of Cantor minimals}

\author{TAKASHI SHIMOMURA}

\address{Nagoya University of Economics, Uchikubo 61-1, Inuyama 484-8504, Japan}
\curraddr{}
\email{tkshimo@nagoya-ku.ac.jp}
\thanks{}



\date{\today}

\dedicatory{}

\commby{}

\begin{abstract}
Downarowicz and Maass (2008) have defined the topological rank for all
 Cantor minimal homeomorphisms.
On the other hand, Gambaudo and Martens (2006) have expressed all Cantor minimal
 continuous surjections as the inverse limits of certain graph coverings.
Using the aforementioned results, we previously extended the notion of
 topological rank to all Cantor minimal
 continuous surjections.
In this paper, we show that taking
 natural extensions of Cantor minimal continuous surjections does not increase their topological ranks.
Further, we apply the result to the minimal symbolic case.
\end{abstract}

\maketitle
\section{Introduction}
In \cite{DOWNAROWICZ_2008FiniteRankBratteliVershikDiagAreExpansive},
 Downarowicz and Maass presented a remarkable result, i.e.,
 a Cantor minimal system of finite topological rank $K > 1$ is expansive.
They used properly ordered Bratteli diagrams and adopted a noteworthy technique.
In \cite{BEZUGLYI_2009AperioSubstSysBraDiag},
 Bezuglyi, Kwiatkowski, and Medynets extended the result
 to non-minimal aperiodic homeomorphic cases.
In this paper, a {\it zero-dimensional system} implies a pair $(X,f)$
 of a compact zero-dimensional metrizable space $X$
 and a continuous surjective map $f : X \to X$.
A zero-dimensional system is a Cantor system if $X$ does not contain any
 isolated point.
In \cite{Shimomura_2014SpecialHomeoApproxCantorSys},
 we showed that every zero-dimensional system is expressed
 as an inverse limit of a sequence of covers of finite directed graphs.
In this paper, instead of the term ``sequence of graph covers,''
 we use the term
 ``graph covering'' or just ``covering'' for short.
In \cite{Gambaudo_2006AlgTopMinCantSets},
 Gambaudo and Martens had already represented
 general Cantor minimal continuous surjections
 by a type of graph covering.
In a previous paper \cite{SHIMOMURA_2015nonhomeomorphic}, we extended
 the definition of topological rank to Cantor minimal continuous surjections
 by applying the Gambaudo--Martens type of graph covering,
 and we showed that
 a Cantor minimal continuous surjection of finite topological
 rank $K > 1$ has a natural extension that is expansive.
We also showed that the two topological ranks are equal to each other
 in the case of homeomorphic Cantor minimal systems.
In addition, we presented some related results.
In this paper, we show that taking natural extensions
of Cantor minimal continuous surjections
 does not increase their topological ranks.
Further, we apply the result to the minimal symbolic case.
\section{Preliminaries}
Let $\Z$ denote the set of all integers;
 $\Nonne$, the set of all non-negative integers;
 and $\Posint$, the set of all positive integers.
In this section,
 to prepare graph coverings of
 the Gambaudo--Martens type,
 we repeat the construction of general graph coverings
 for general zero-dimensional systems originally given in
 \cite[\cref{sec:Bratteli}]{Shimomura_2014SpecialHomeoApproxCantorSys}.
For $m \ge n$, we denote $[n,m] := \seb n, n+1,\dotsc, m\sen$.
A pair $G = (V,E)$
 consisting of a finite set $V$ and a relation $E \subseteq V \times V$ on $V$
 can be considered as a directed graph with vertices $V$
 and an edge from $u$ to $v$ when $(u,v) \in E$.
Unlike the case of Bratteli diagrams, which are well known and defined in
 \cref{sec:Bratteli},
 multiple edges from a vertex $u$ to $v$ are not permitted.
Here, we note that
 the expression $(V,E)$ is also used to represent a Bratteli diagram.
If we write ``graph $G$,'' ``graph $G = (V,E)$,'' or
 ``surjective directed graph $G = (V,E)$,''
 we imply a finite directed graph.
When the expression $(V,E)$ represents a Bratteli diagram,
 we explicitly write ``Bratteli diagram $(V,E)$.''
%
%
\begin{nota}
In this paper,
 we assume that a finite directed graph $G$ is a surjective relation,
 i.e., for every vertex $v \in V$, there exist edges $(u_1,v),(v,u_2) \in E$.
\end{nota}
For directed graphs $G_i = (V_i,E_i)$ with $i = 1,2$,
 a map $\fai : V_1 \to V_2$ is said to be a {\it graph homomorphism}
 if for every edge $(u,v) \in E_1$,
 it follows that $(\fai(u),\fai(v)) \in E_2$.
In this case, we write $\fai : G_1 \to G_2$.
For a graph homomorphism $\fai : G_1 \to G_2$,
 we say that $\fai$ is {\it edge-surjective}
if $\fai(E_1) = E_2$.
%
%
%
%
%
Suppose that a graph homomorphism $\fai : G_1 \to G_2$
 satisfies the following condition:
\[(u,v),(u,v') \in E_1 \text{ implies that } \fai(v) = \fai(v').\]
In this case, $\fai$ is said to be {\it \pdirectional}.
Suppose that a graph homomorphism $\fai$ satisfies
 both of the following conditions:
\[(u,v),(u,v') \in E_1 \text{ implies that } \fai(v) = \fai(v') \myand \]
\[(u,v),(u',v) \in E_1 \text{ implies that } \fai(u) = \fai(u').\]
Then, $\fai$ is said to be {\it \bidirectional}.\\
%
%
\begin{defn}\label{defn:cover}
A graph homomorphism $\fai : G_1 \to G_2$ is called a \textit{cover}
 if it is a \pdirectional\ \tesgh.
\end{defn}
For a sequence $G_1 \getsby{\fai_2} G_2 \getsby{\fai_3} \dotsb$
 of graph homomorphisms
 and $m > n$, we write
 $\fai_{m,n} := \fai_{n+1} \circ \fai_{n+2} \circ \dotsb \circ \fai_{m}$.
Then, $\fai_{m,n}$ is a graph homomorphism.
If all ${\fai_i}$ $(i \beposint)$ are edge-surjective,
 then every $\fai_{m,n}$ is edge-surjective.
If all ${\fai_i}$ $(i \beposint)$ are covers, every $\fai_{m,n}$ is a cover.
%
%
Let $G_0 := \left(\seb v_0 \sen, \seb (v_0,v_0) \sen \right)$
 be a singleton graph.
For a sequence of graph covers
 $G_1 \getsby{\fai_2} G_2 \getsby{\fai_3} \dotsb$, we
 attach the singleton graph $G_0$ at the head.
We refer to a sequence of graph covers
 $G_0 \getsby{\fai_1} G_1 \getsby{\fai_2} G_2 \getsby{\fai_3} \dotsb$
 as a \textit{graph covering} or just a \textit{covering}.
%
%
%
%
Let us express the directed graphs as $G_i = (V_i,E_i)$ for $i \benonne$.
Define
\[V_{\Gcal} := \seb (x_0,x_1,x_2,\dotsc)
 \in \prod_{i = 0}^{\infty}V_i \mid x_{i} = \fai_{i+1}(x_{i+1})
 \text{ for all } i \benonne \sen \text{ and}\]
\[E_{\Gcal} := \seb (x,y) \in V_{\Gcal} \times V_{\Gcal} \mid
 (x_i,y_i) \in E_i \text{ for all } i \benonne\sen,\]
each equipped with the product topology. Further, each $V_i$ is equipped with the discrete topology.
\begin{nota}\label{nota:opensets-of-vertices}
Let $X= V_{\Gcal}$,
 and let us define a map $f:X \to X$ by $f(x) =y$
 if and only if $(x, y) \in E_\Gcal$.
For each $n \benonne$,
 the projection from $V_{\Gcal}$ to $V_n$ is denoted by $\fai_{\infty,n}$.
For $v \in V_n$, we denote a clopen set $U(v) := \fai_{\infty}^{-1}(v)$.
For a subset $V \subset V_n$,
 we denote a clopen set $U(V) := \bigcup_{v \in V}U(v)$.
\end{nota}
%
%
We can state the following:
\begin{thm}[
{\cite[Theorem 3.9 and Lemma 3.5]{Shimomura_2014SpecialHomeoApproxCantorSys}}
]\label{thm:0dim=covering}
Let $\Gcal$ be a covering
 $\covrepa{G}{\fai}$.
Let $X=V_\Gcal$ and let us define $f:X \to X$ as above.
Then,
 $f$ is a continuous surjective mapping and $(X,f)$ is
 a zero-dimensional system.
Conversely, every zero-dimensional system can be written in this manner.
Furthermore,
 if all $\fai_n$ are \bidirectional, then this zero-dimensional system is a
 homeomorphism and every compact zero-dimensional homeomorphism
 is written in this manner.
\end{thm}
%
%
We write $(V_{\Gcal},E_{\Gcal})$ as $G_{\infty}$.
%
%
%
Take a subsequence $n_0 = 0 < n_1 < n_2 < \dotsb$.
Then, we can get essentially the same covering
\[ G_0 \getsby{\fai_{n_1,0}} G_{n_1} \getsby{\fai_{n_2,n_1}} G_{n_2} \dotsb.\]
It is evident that the new covering produces
 a naturally topologically conjugate zero-dimensional  system.
Following the terminology in the theory of Bratteli--Vershik systems,
 we refer to this procedure as
 {\it telescoping}.
%
%
%
%
%
%
%
%
%
\begin{nota}
Let $G = (V,E)$ be a surjective directed graph.
A sequence of vertices $(v_0,v_1,\dotsc,v_l)$ of $G$ is said to be
 a \textit{walk} of length $l$ if $(v_i, v_{i+1}) \in E$ for all $0 \le i < l$.
We denote $l(w) := l$.
We say that a walk $w = (v_0,v_1,\dotsc,v_l)$ is a {\it path}\/
 if $v_i$ $(0 \le i \le l)$ are mutually distinct.
A walk $c = (v_0,v_1,\dotsc,v_l)$ is said to be a {\it cycle}\/ of period $l$
 if $v_0 = v_l$,
 and a cycle $c = (v_0,v_1,\dotsc,v_l)$
 is a \textit{circuit} of period $l$
 if $v_i$ $(0 \le i < l)$ are mutually distinct.
Further, a circuit $c$ and a path $p$ are considered to be subgraphs of $G$
 with period $l(c)$ and length $l(p)$, respectively.
Let $\sC(G)$ be the set of all circuits of $G$.
For a walk $w = (v_0,v_1,\dotsc,v_l)$,
 we define $V(w) := \seb v_i \mid 0 \le i \le l \sen$
 and $E(w) := \seb (v_i,v_{i+1}) \mid 0 \le i < l \sen$.
For a subgraph $G'$ of $G$,
 we define $V(G')$ and $E(G')$ in the same manner;
 in particular, $V(G) = V$ and $E(G) = E$.
\end{nota}
Next, we introduce a proposition that describes a condition of minimality
 of the inverse limit of a graph covering.
\begin{prop}\label{prop:minimal}
Let $\covrepa{G}{\fai}$ be a covering.
Then, the resulting zero-dimensional system $G_{\infty}$ is minimal
 if and only if for all $n \ge 0$,
 there exists an $m > n$ such that every $c \in \sC(G_m)$ satisfies
 $V(\fai_{m,n}(c)) = V(G_n)$.
\end{prop}
\begin{proof}
From \cite[(a),(d),(e) of Theorem 3.5]{Shimomura_2014ergodic},
 the conclusion is obvious.
\end{proof}
\section{Bratteli--Vershik systems}\label{sec:Bratteli}
\begin{defn}
A \textit{Bratteli diagram} is an infinite directed graph $(V,E)$,
 where $V$ is the vertex
set and $E$ is the edge set.
These sets are partitioned into non-empty disjoint finite sets
$V = V_0 \cup V_1 \cup V_2 \cup \dotsb$ and $E = E_1 \cup E_2 \cup \dotsb$,
 where $V_0 = \seb v_0 \sen$ is a one-point set.
Each $E_n$ is a set of edges from $V_{n-1}$ to $V_n$.
Therefore, there exist two maps $r,s : E \to V$ such that $r:E_n \to V_n$
 and $s : E_n \to V_{n-1}$ for $n \bpi$,
 i.e., the range map and the source map, respectively.
Moreover, $s^{-1}(v) \nekuu$ for all $v \in V$ and
$r^{-1}(v) \nekuu$ for all $v \in V \setminus V_0$.
We say that $u \in V_{n-1}$ is connected to $v \in V_{n}$ if there
 exists an edge $e \in E_n$ such that $s(e) = u$ and $r(e) = v$.
Unlike the case of graph coverings,
 multiple edges between $u$ and $v$ are permitted.
The {\it rank $K$} of a Bratteli diagram is defined as
 $K := \liminf_{n \to \infty}\hash V_n$,
 where $\hash V_n$ is the number of elements in $V_n$.
\end{defn}
%
%
Let $(V,E)$ be a Bratteli diagram and $m < n$ be non-negative integers.
We define
\[E_{m,n} :=
 \seb p \mid
 p \text{ is a path from a } u \in V_m \text{ to a } v \in V_n \sen.
\]
Then, we can construct a new Bratteli diagram $(V',E')$ as follows:
\[ V' := V_0 \cup V_1 \cup \dotsb \cup V_m \cup V_n \cup V_{n+1} \cup \dotsb \]
\[ E' :=
 E_1 \cup E_2 \cup \dotsb \cup E_m \cup E_{m,n} \cup E_{n+1} \cup \dotsb. \]
The source map and the range map are also defined naturally.
This procedure is called \textit{telescoping}.
\begin{defn}
A Bratteli diagram is called \textit{simple}
 if, after (at most countably many) telescopings,
 we get that for all $n \ge 0$,
 all pairs of vertices $u \in V_n$ and $v \in V_{n+1}$ are joined
 by at least one edge.
\end{defn}

\begin{defn}
Let $(V,E)$ be a Bratteli diagram such that
$V = V_0 \cup V_1 \cup V_2 \cup \dotsb$ and $E = E_1 \cup E_2 \cup \dotsb$
 are the partitions,
 where $V_0 = \seb v_0 \sen$ is a one-point set.
Let $r,s : E \to V$ be the range map and the source map, respectively.
We say that $(V,E,\le)$ is an \textit{ordered} Bratteli diagram if
 the partial order $\le$ is defined on $E$ such that
 $e, e' \in E$ are comparable if and only if $r(e) = r(e')$.
 In other words, we have a linear order on each set $r^{-1}(v)$ with $v \in \Vp$.
The edges $r^{-1}(v)$ are numbered from $1$ to $\hash(r^{-1}(v))$.
\end{defn}
Let $n > 0$
 and $e = (e_n,e_{n+1},e_{n+2},\dotsc), e'=(e'_n,e'_{n+1},e'_{n+2},\dotsc)$
 be cofinal paths from the vertices of $V_{n-1}$, which might be different.
We obtain the lexicographic order $e < e'$ as follows:
\[\myif k \ge n \text{ is the largest number such that } e_k \ne e'_k,
 \mythen e_k < e'_k.\]
\begin{defn}
Let $(V,E,\le)$ be an ordered Bratteli diagram.
Let $E_{\max}$ and $E_{\min}$ denote the sets of maximal and minimal
edges, respectively.
A path is maximal (resp. minimal) if all the edges constituting the
path are elements of $E_{\max}$ (resp. $E_{\min}$).
\end{defn}

\begin{defn}
An ordered Bratteli diagram is properly ordered if it is
simple and if it has a unique maximal path and a unique minimal path,
 denoted respectively by
$x_{\max}$ and $x_{\min}$.
\end{defn}

\begin{defn}[Vershik map]
Let $(V,E,\le)$ be a properly ordered Bratteli diagram.
Let
\[E_{0,\infty} := \seb (e_1,e_2,\dotsc) \mid r(e_i) = s(e_{i+1})
 \myforall i \ge 1 \sen,\]
with the subspace topology of the product space $\prod_{i = 1}^{\infty}E_i$.
We can define a \textit{Vershik map}
 $\phi : E_{0,\infty} \to E_{0,\infty}$ as follows:

\noindent If $e = (e_1,e_2,\dotsc) \ne x_{\max}$,
 then there exists the least $n \ge 1$ such that
 $e_n$ is not maximal in $r^{-1}(r(e_n))$.
Then, we can select the least $f_n > e_n$ in $r^{-1}(r(e_n))$.
Let $v_{n-1} = s(f_n)$.
Then, it is easy to obtain the unique least path $(f_1,f_2,\dotsc,f_{n-1})$
 from $v_0$ to $v_{n-1}$.
We define
\[\phi(e) := (f_1,f_2,\dotsc,f_{n-1},f_n,e_{n+1},e_{n+2},\dotsc).\]
Further, we define $\phi(x_{\max}) = x_{\min}$.
The map $\phi : E_{0,\infty} \to E_{0,\infty}$ is called
 the \textit{Vershik map}.
\end{defn}
On the basis of a previously introduced theorem \cite[Theorem 4.7]{HERMAN_1992OrdBratteliDiagDimGroupTopDyn},
 we can find a correspondence
 that a properly ordered Bratteli diagram brings about the Vershik map
 that is a minimal homeomorphic zero-dimensional system.
 Conversely, a minimal homeomorphic zero-dimensional system
 is represented as the Vershik map of a properly ordered Bratteli diagram.
In \cite{DOWNAROWICZ_2008FiniteRankBratteliVershikDiagAreExpansive},
 Downarowicz and Maass introduced the topological rank for
 Cantor minimal homeomorphisms.
\begin{defn}
Let $(X,f)$ be a Cantor minimal homeomorphism.
Then, the topological rank of $(X,f)$ is $1 \le K \le \infty$
 if it has a Bratteli--Vershik representation with a Bratteli diagram
 of rank $K$,
 and $K$ is the minimum of such numbers.
\end{defn}%
\section{Covering of Gambaudo--Martens Type}\label{sec:gambaudo-martens}
%
%
In this section, we introduce a covering of the Gambaudo--Martens type
 and define the topological rank for all Cantor minimal continuous
 surjections.
Then, we prepare the proof of
 our main result.
In \cite[Theorem 2.5]{Gambaudo_2006AlgTopMinCantSets},
 Gambaudo and Martens showed that every Cantor minimal system is
an inverse limit of a special type of graph covering.
In our context, their construction of a graph covering is as follows.
 Let $\covrepa{G}{\fai}$ be a graph covering.
As usual, we assume that $G_0$ is a singleton graph
 $(\seb v_0 \sen,\seb (v_0,v_0) \sen)$.
We shall construct graphs $G_n$ with an $n \ge 1$ such that
 there exist a unique vertex $v_{n,0}$
 and a finite number of circuits $c_{n,i}~(1 \le i \le r_n)$
 that start and end at $v_{n,0}$.
 Roughly, if two circuits meet at a vertex, then the remaining circuits merge until they reach the end.
\begin{defn}\label{defn:GM-covering}
We say that
 a covering
 $\covrepa{G}{\fai}$
 is of the \textit{Gambaudo--Martens type} if for each $n > 0$, there exist
 a vertex $v_{n,0}$, a finite number of circuits
 $c_{n,i}~(1 \le i \le r_n)$, and a covering map $\fai_n$ such that
\enumb
\item $c_{n,i}$ can be
 written as
 $(v_{n,0}=v_{n,i,0},v_{n,i,1},v_{n,i,2},\dotsc,v_{n,i,l(n,i)} = v_{n,0})$
 with $l(n,i) \ge 1$,
\item $\bigcup_{i = 1}^{r_n}E(c_{n,i}) = E(G_n)$,
\item if $v_{n,i,j} = v_{n,i',j'}$ with $j,j' \ge 1$, then
 $v_{n,i,j+k} = v_{n,i',j'+k}$
 for $k = 0,1,2,\dotsc$, until $j+k = l(n,i)$ and $j'+k = l(n,i')$
 at the same time,
\item $\fai_n(v_{n,0}) = v_{n-1,0}$ for all $n \bpi$, and
\item\label{item:pdirectional}
 $\fai_n(v_{n,i,1}) = v_{n-1,1,1}$ for all $n \bpi$ and $1 \le i \le r_n$.
\enumn
We say that a covering of this type is a \textit{GM-covering} for short.
We denote $\sC_n := \sC(G_n) = \seb c_{n,i} \mid 1 \le i \le r_n \sen$.
A GM-covering is said to be {\it simple} if
 for all $n > 0$, there exists an $m > n$ such that
 for each $1 \le i \le r_m$, $E(\fai_{m,n}(c_{m,i})) = E(G_n)$.
By \cref{prop:minimal},
 this condition makes the resulting zero-dimensional system minimal.
If we want to avoid the case in which the resulting zero-dimensional system
has an isolated point, we have to add the following condition:
 for every $n \ge 1$ and every vertex $v$ of $G_n$,
 there exist an $m > n$ and distinct vertices $u_1,u_2$ of $G_m$ such that
 $\fai_{m,n}(u_1) = \fai_{m,n}(u_2) =v$.
The \textit{rank} of a GM-covering is
 the integer $1 \le K \le \infty$ defined by $K := \liminf_{n \to \infty}r_n$.
\end{defn}
\begin{rem}
For $n \ge 0$ and $1 \le i \le r_n$, we can write
\[\fai_n(c_{n,i})
 = c_{n-1,a(n,i,1)}c_{n-1,a(n,i,2)}\dotsb c_{n-1,a(n,i,k(n,i))},\]
such that $a(n,i,1) = 1$ for all $i$ with $1 \le i \le r_n$.

\end{rem}

%
%
\begin{nota}\label{nota:GMstep1}
By telescoping,
 we can add the following condition to a simple GM-covering:
for every $n \bpi$ and every  $i~(1 \le i \le r_n)$,
 $E(\fai_n(c_{n,i})) = E(G_{n-1})$.
Hereafter,
 if we say that a GM-covering is simple, we assume that this condition
 is satisfied.
\end{nota}
\begin{thm}[Gambaudo and Martens, \cite{Gambaudo_2006AlgTopMinCantSets}]
A zero-dimensional system is minimal (not necessarily homeomorphic)
 if and only if it is represented as the inverse limit of a simple GM-covering.
\end{thm}
\begin{proof}
See \cite[Theorem 2.5]{Gambaudo_2006AlgTopMinCantSets}.
\end{proof}
As an analogue of topological rank for Cantor minimal homeomorphisms,
 we say that a minimal zero-dimensional system has
 \textit{topological rank} $K$
 if there exists a simple GM-covering of rank $K$,
 and $K$ is the minimum of such numbers
 (see \cite{DOWNAROWICZ_2008FiniteRankBratteliVershikDiagAreExpansive}).
In \cite{DOWNAROWICZ_2008FiniteRankBratteliVershikDiagAreExpansive},
 it was shown that a Cantor minimal homeomorphism with finite topological
 rank $K > 1$ is expansive,
 i.e., topologically conjugate to a minimal two-sided subshift.
%
The remainder of this section is devoted to preparing the statement of
 our main result and its proof.
Suppose that a simple GM-covering
 $\covrepa{G}{\fai}$
 produces a minimal zero-dimensional system $G_{\infty}$.
We write $G_{\infty}  = (X,f)$.
We assume that $(X,f)$ is not a single periodic orbit.
Then, because of minimality, $(X,f)$ is a Cantor system
 and has no periodic orbits.
Therefore, the minimal length of the circuits of $G_n$ becomes infinity, i.e.,
 we get $l(n,i) \to \infty$ uniformly as $n \to \infty$.
\begin{nota}\label{nota:natural-extension}
For $(X,f)$, we construct the natural extension $(\hX_f,\sigma)$ as follows:
\itemb
\item $\hX_f :=
 \seb (\dotsc,x_{-1},x_0,x_{1},x_{2}, \dotsc) \in X^{\Z}
 \mid f(x_{i}) = x_{i+1} \ \ \myforall\  i \bi \sen$;
\item for $\hx = (\dotsc,x_{-1},x_0,x_{1},x_{2}, \dotsc) \in \hX_f$,
 $\sigma$ shifts $\hx$ to the left,
 i.e., $(\sigma(\hx))_i = x_{i+1}$ for all $i \bi$.
\itemn
\end{nota}
It is easy to check that if $(X,f)$ is minimal,
 then $(\hX_f,\sigma)$ is minimal.
For an $\hx \in \hX_f$ and an $i \bi$, we denote $\hx(i) := x_i$.
Then, $(\sigma(\hx))(i) = \hx(i+1)$ for all $i \bi$.
We use many notations and concepts from
 \cite{DOWNAROWICZ_2008FiniteRankBratteliVershikDiagAreExpansive}.
For every $\hx \in \hX_f$ and $i \bi$,
 there exists a unique $u_{n,i} \in V(G_n)$ such that $x_i \in U(u_{n,i})$.
Therefore, for each $n \bni$, a unique sequence
 $\hx|_{n} = (\dotsc,u_{n,-2},u_{n,-1},u_{n,0},u_{n,1},\dotsc)$
 of vertices of $G_n$ is defined such that
 $x_i \in U(u_{n,i})$ for all $i \bi$.
We write $\hx|_{n}(i)  := u_{n,i}$ for all $i \bi$.
Although the vertex $u_{n,i}$ is uniquely determined for each $\hx$,
 $n \ge 0$ and $i \bi$,
 the circuit $c_{n,t}$ with $u_{n,i} \in V(c_{n,t})$ may not be unique.
Nevertheless, if $x_i \in U(v_{n,0})$ for some $i \bi$,
 then there exists a unique $t~(1 \le t \le r_n)$ such that
 $x_{i+1} \in U(v_{n,t,1})$;
 therefore, $x_{i+j} \in U(v_{n,t,j})$ for all $0 \le j \le l(n,t)$.
%
%
\begin{figure}
\begin{center}\leavevmode 
\xy
(0,18)*{}; (100,18)*{} **@{-},
 (48,16)*{c_{n,1}
 \hspace{6mm} c_{n,3} \hspace{6mm}
 \hspace{12mm} c_{n,1} \hspace{8mm}
 \hspace{5mm} c_{n,3} \hspace{5mm}
 \hspace{6mm} c_{n,2} \hspace{6mm}
 \hspace{4mm} c_{n,1}
 \hspace{2mm} },
(7,18)*{}; (7,14)*{} **@{-},
(28,18)*{}; (28,14)*{} **@{-},
(49,18)*{}; (49,14)*{} **@{-},
(70,18)*{}; (70,14)*{} **@{-},
(84,18)*{}; (84,14)*{} **@{-},
(0,14)*{}; (100,14)*{} **@{-},
 (55,12)*{c_{n+1,5}
 \hspace{36mm} c_{n+1,1} \hspace{6mm}
 \hspace{20mm} c_{n+1,3} \hspace{8mm} },
(28,14)*{}; (28,10)*{} **@{-},
(84,14)*{}; (84,10)*{} **@{-},
(0,10)*{}; (100,10)*{} **@{-},
\endxy
\end{center}
\caption{$n$th and $(n+1)$th rows
 of a linked array system with cuts.}\label{fig:array-system}
\end{figure}
\begin{nota}
We write this $t$ as $t(\hx,n,i)$, and $c_{n,t}$ as $c(\hx,n,i)$,
 for all $n \ge 0$ and $i \bi$.
\end{nota}
Let $k(0) \bi$ such that $x_{k(0)} \in U(v_{n,0})$, and
 let $k(1) > k(0)$ be the least $k > k(0)$ such that $x_k \in U(v_{n,0})$.
Then,
 we combine the interval $u_{n,k(0)},u_{n,k(0)+1},\dotsc,u_{n,k(1)-1}$
 with the unique circuit $c(\hx,n,i)$ with $k(0) \le i < k(1)$.
Thus, we obtain a sequence of $c_{n,i}$s, and we denote it as $\hx[n]$.
We write $\hx[n](i) = c(\hx,n,i)$ for all $n \ge 0$ and $i \bi$.
To mark the beginning of a circuit, it is sufficient
 to change $c(\hx,n,i)$ to $\check{c}(\hx,n,i)$ for each $i \bi$ with $u_{n,i} = v_{n,0}$.
Nevertheless,
 as in \cite{DOWNAROWICZ_2008FiniteRankBratteliVershikDiagAreExpansive},
 for each sequence $\hx[n]$ of circuits of $G_n$ $(n > 0)$,
 instead of changing the symbol,
 we make an \textit{$n$-cut at position $i \bi$}
 just before $i \bi$ with $\hx|_n(i) = v_{n,0}$, i.e.,
 if there exists an $n$-cut at position $i \bi$,
 then $c(\hx,n,i-1)$ and $c(\hx,n,i)$ are separated by the cut
 (see \cref{fig:array-system}).
Note that we can recover the sequence of
 vertices of $G_n$ from $\hx[n]$.
An $\hx[0]$ becomes just a sequence of $e_0 := (v_0,v_0)$.
%
%
For an interval $[n,m]$ with $m > n$, the combination of rows $\hx|_{n'}$
 with $n \le n' \le m$ is denoted as $\hx|_{[n,m]}$, and
 the combination of rows $\hx[n']$
 with $n \le n' \le m$ is denoted as $\hx[n,m]$.
The \textit{array system} of $\hx$ is the infinite combination
 $\hx|_{[0,\infty)}$ of all rows $\hx|_n$ $0 \le n < \infty$.
The \textit{linked array system} of $\hx$ is the infinite combination
 $\hx[0,\infty)$ of all rows $\hx[n]$ ($0 \le n < \infty$)
 (see \cref{fig:array-system-2}).
Note that from the information of $\hx[0,\infty)$,
 we can recover $\hx|_{[0,\infty)}$
 and also identify $\hx$ itself.
If each circuit of $G_n$ is considered to be just an alphabet,
 then for $n \ge 0$ and $I < J$,
 we can consider a finite sequence of circuits of $G_n$,
\[\hx[n](I),\hx[n](I+1), \dotsc, \hx[n](J),\]
even if the completion of the circuits is cut off
 at the right or left end in the above sequence.
\begin{defn}\label{defn:array-systems}
Let $\ddX_f := \seb \hx|_{[0,\infty)} \mid \hx \in \hX_f \sen$
 be a set of sequences of
 symbols that are vertices of $G_n$ $(0 \le n < \infty)$.
The topology is generated by \textit{cylinders}
 such that for $\hx \in \hX_f$ and $N,I > 0$,
\[\ddC(\hx,N,I) :=
 \seb \hy|_{[0,\infty)} \mid \hy \in \hX_f,\
 \hy|_{[0,N]}(i) = \hx|_{[0,N]}(i) \myforall
 i \mywith -I \le i \le I \sen.\]
The shift map $\sigma : \ddX_f \to \ddX_f$ is defined as above.
Then, $(\ddX_f,\sigma)$ is a zero-dimensional system, and we refer to it
 as an \textit{array system} of $(\hX_f,\sigma)$.
Let $\barX_f := \seb \hx[0,\infty) \mid \hx \in \hX_f \sen$
 be a set of sequences of
 symbols that are circuits of $G_n$ $(0 \le n < \infty)$.
The topology is generated by \textit{cylinders}
 such that for $\hx \in \hX_f$ and $N,I > 0$,
\[\barC(\hx,N,I) :=
 \seb \hy[0,\infty) \mid \hy \in \hX_f,~\hy[0,N](i) = \hx[0,N](i) \myforall
 i \mywith -I \le i \le I \sen.\]
The shift map $\sigma : \barX_f \to \barX_f$ is defined as above.
Then, $(\barX_f,\sigma)$ is a zero-dimensional system, and we refer to it
 as a {\it linked array system} of $(\hX_f,\sigma)$.
\end{defn}

\begin{rem}
Clearly, $(\ddX_f,\sigma)$ is topologically conjugate to $(\hX_f,\sigma)$.
Moreover,
 because $(\ddX_f,\sigma)$ has a continuous factor map to $(\barX_f,\sigma)$
 and it is bijective as described above,
 $(\barX_f,\sigma)$ is also topologically
 conjugate to $(\hX_f,\sigma)$.
\end{rem}
%
%
%
%
The row $\hx[n]$ is precisely separated into circuits by the cuts.
Note that for $m > n$, if there exists an $m$-cut at position $k$, then
 there exists an $n$-cut at position $k$.
%
%
%
%
\begin{figure}
\begin{center}\leavevmode 
\xy
(0,22)*{}; (100,22)*{} **@{-},
 (50,20)*{v_0\phantom{{}_{{},f}}\ v_0\phantom{{}_{{},l}}
\ v_0\phantom{{}_{{},l}}\ v_0\phantom{{}_{{},l}}
\ v_0\phantom{{}_{{},f}}\ v_0\phantom{{}_{{},l}}
\ v_0\phantom{{}_{{},l}}\ v_0\phantom{{}_{{},l}}
\ v_0\phantom{{}_{{},f}}\ v_0\phantom{{}_{{},l}}
\ v_0\phantom{{}_{{},l}}\ v_0\phantom{{}_{{},l}}
\ v_0\phantom{{}_{{},f}}\ v_0\phantom{{}_{{},l}}},
(0,22)*{}; (0,18)*{} **@{-},
(7,22)*{}; (7,18)*{} **@{-},
(14,22)*{}; (14,18)*{} **@{-},
(21,22)*{}; (21,18)*{} **@{-},
(28,22)*{}; (28,18)*{} **@{-},
(35,22)*{}; (35,18)*{} **@{-},
(42,22)*{}; (42,18)*{} **@{-},
(49,22)*{}; (49,18)*{} **@{-},
(56,22)*{}; (56,18)*{} **@{-},
(63,22)*{}; (63,18)*{} **@{-},
(70,22)*{}; (70,18)*{} **@{-},
(77,22)*{}; (77,18)*{} **@{-},
(84,22)*{}; (84,18)*{} **@{-},
(91,22)*{}; (91,18)*{} **@{-},
(98,22)*{}; (98,18)*{} **@{-},
(0,18)*{}; (100,18)*{} **@{-},
 (49,16)*{c_{1,3}
 \hspace{8mm} c_{1,3} \hspace{6mm}
 \hspace{11mm} c_{1,1} \hspace{8mm}
 \hspace{7mm} c_{1,3} \hspace{5mm}
 \hspace{6mm} c_{1,2} \hspace{6mm}
 \hspace{4mm} c_{1,1}
 \hspace{2mm} },
(7,18)*{}; (7,14)*{} **@{-},
(28,18)*{}; (28,14)*{} **@{-},
(49,18)*{}; (49,14)*{} **@{-},
(70,18)*{}; (70,14)*{} **@{-},
(84,18)*{}; (84,14)*{} **@{-},
(0,14)*{}; (100,14)*{} **@{-},
(56,12)*{c_{2,3} \hspace{0.7mm}\
 \hspace{32mm} c_{2,1} \hspace{27mm}
 \hspace{10mm} c_{2,3}
 \hspace{2mm} },
(28,14)*{}; (28,10)*{} **@{-},
(84,14)*{}; (84,10)*{} **@{-},
(0,10)*{}; (100,10)*{} **@{-},
(36,8)*{c_{3,3} \hspace{26mm}
 \hspace{15mm} c_{3,1}},
(28,10)*{}; (28,6)*{} **@{-},
(0,6)*{}; (100,6)*{} **@{-},
(50,4)*{\vdots},
\endxy
\end{center}
\caption{The first 4 rows of a linked array system.}\label{fig:array-system-2}
\end{figure}
%
%
%
%
%
For each circuit $c_{n,i}$, we can determine a series of circuits
 by
 $\fai_n(c_{n,i}) = c_{n-1,1}c_{n-1,a(n,i,2)}\dotsb c_{n-1,a(n,i,k(n,i))}$.
Furthermore,
 each $c_{n-1,a(n,i,j)}$ determines a series of circuits by the map
 $\fai_{n-1}$.
Thus, we can determine a set of circuits arranged in a square form as in
 \cref{2-symbol}.
Following \cite{DOWNAROWICZ_2008FiniteRankBratteliVershikDiagAreExpansive},
 this form is said to be the {\it $n$-symbol} and denoted by $c_{n,i}$.
For $m < n$,
 the projection $c_{n,i}[m]$ that is a finite sequence of circuits of $G_m$
 is also defined.
\begin{figure}
\begin{center}\leavevmode 
\xy
(28,32)*{}; (84,32)*{} **@{-},
(57,30)*{v_0\phantom{{}_{{},l}}\ v_0\phantom{{}_{{},l}}
\ v_0\phantom{{}_{{},l}}\ v_0\phantom{{}_{{},l}}
\ v_0\phantom{{}_{{},f}}\ v_0\phantom{{}_{{},l}}
\ v_0\phantom{{}_{{},l}}\ v_0\phantom{{}_{{},l}}},
(28,32)*{}; (28,28)*{} **@{-},
(35,32)*{}; (35,28)*{} **@{-},
(42,32)*{}; (42,28)*{} **@{-},
(49,32)*{}; (49,28)*{} **@{-},
(56,32)*{}; (56,28)*{} **@{-},
(63,32)*{}; (63,28)*{} **@{-},
(70,32)*{}; (70,28)*{} **@{-},
(77,32)*{}; (77,28)*{} **@{-},
(84,32)*{}; (84,28)*{} **@{-},
(28,28)*{}; (84,28)*{} **@{-},
 (60,26)*{
 \hspace{7mm} c_{1,1} \hspace{8mm}
 \hspace{7mm} c_{1,3} \hspace{6mm}
 \hspace{6mm} c_{1,2} \hspace{8mm}
 \hspace{2mm} },
(28,28)*{}; (28,24)*{} **@{-},
(49,28)*{}; (49,24)*{} **@{-},
(70,28)*{}; (70,24)*{} **@{-},
(84,28)*{}; (84,24)*{} **@{-},
(28,24)*{}; (84,24)*{} **@{-},
(60,22)*{
 \hspace{23mm} c_{2,1} \hspace{23mm}
 \hspace{2mm} },
(28,24)*{}; (28,20)*{} **@{-},
(84,24)*{}; (84,20)*{} **@{-},
(28,20)*{}; (84,20)*{} **@{-},
\endxy
\end{center}
\caption{The $2$-symbol corresponding to the circuit $c_{2,1}$
 of Figure \ref{fig:array-system-2}.}\label{2-symbol}
\end{figure}
The set $X_n := \seb \hx[n] \mid \hx \in \hX_f \sen$ is a two-sided subshift
 of the finite set
 $\sC_n \cup
 \seb \check{c}_{n,1}, \check{c}_{n,2}, \dotsc, \check{c}_{n,r_n} \sen$.
The factoring map is denoted by $\pi_n : \hX_f \to X_n$,
 and the shift map is denoted by $\sigma_n : X_n \to X_n$.
We simply write $\sigma = \sigma_n$ for all $n$
 if there is no confusion.
%

%
%
Next, we wish to briefly recall the construction of the array system
 in \cite{DOWNAROWICZ_2008FiniteRankBratteliVershikDiagAreExpansive}.
Let $(V,E,\ge)$ be a properly ordered Bratteli diagram
 with the Vershik map
 $\phi : E_{0,\infty} \to E_{0,\infty}$.
Let $x \in E_{0,\infty}$.
We write $\phi^i(x) = (e_{1,i},e_{2,i},\dotsc)$ for all $i \bi$.
Then, we can construct a sequence $v_{n,i} = s(e_{n+1,i})$ for
 all $n \bni$ and $i \bi$.
For each $n \bni$, we denote $v_{x}[n] := (v_{n,i})_{i \bi}$ and
 the combination of these lines as $v_x := v_x[0,\infty)$.
For each $n \bni$, we make an $n$-cut by the following argument.
For each $v \in V_n$, we define
 $P_v := \seb (e_1,e_2,\dotsc,e_{n-1}) \mid
 r(e_{n-1}) = v \sen$.
We make an $n$-cut just before
 $(e_1,e_2,\dotsc,e_{n-1},\dotsc) \in E_{0,\infty}$
 such that $(e_1,e_2,\dotsc,e_{n-1}) \in P_v$ is minimal.
Thus, as in the case of the previous argument of a GM-covering,
 $v_x[n]$ gets $n$-cuts.
We define $Y(V,E,\ge):= \seb v_x \mid x \in E_{0,\infty} \sen$.
Let $Y = Y(V,E,\ge)$
and let $\sigma : Y \to Y$ be the left shift.
Then,
 $(E_{0,\infty},\phi)$ is naturally topologically conjugate
 to $(Y,\sigma)$.
%
%
%
%

%
%
%
%
%
%
\section{Main Theorem.}
In this section, we state our main result and prove the theorem.
\begin{thm}[Main Result]
Let $(X,f)$ be a minimal (not necessarily homeomorphic)
 zero-dimensional system with topological rank $K \ge 1$.
Then, its natural extension $(\hX_f,\sigma)$ has topological rank $\le K$.
\end{thm}
\begin{proof}
By the assumption, there exists a simple GM-covering
 $\Gcal : \covrepa{G}{\fai}$ of
 rank $K$ such that $G_{\infty}$ is topologically conjugate to $(X,f)$.
Therefore, we assume that $G_{\infty} = (X,f)$.
In \cref{sec:gambaudo-martens}, we defined a linked array system
 $(\barX_f,\sigma)$.
We have stated that
\[\fai_{n}(c_{n,i}) = c_{n-1,1}c_{n-1,a(n,i,2)}\dotsb c_{n-1,a(n,i,k(n,i))}
 \text{ for each } 1 \le i \le r_n.\]
By telescoping, we can assume that $k(n,i) > 2$ for all $1 \le i \le r_n$
 and $a(n,i,2)$ is independent of $i$.
We write $a(n) := a(n,i,2)$.
Thus, we can write
 $\fai_n(c_{n,i}) = c_{n-1,1} c_{n-1,a(n)} d_{n-1,i}$.
From here, we make another symbolic linked array system $(Y,\sigma)$.
Later, we check whether the symbolic linked array system is actually linked with
 a properly ordered Bratteli diagram.
To make another array system, let $s_{n} = l(c_{n-1,1})$ for all $n \ge 2$.
We denote $S(n) := \sum_{i = 2}^{n}s(i)$.
Let $\barx \in \barX_f$.
For each $n \bpi$, we make another sequence $\bary[n](i) := \barx[n](i+S(n))$,
 i.e., we make different slides for each line $\barx[n]$ ($n \bpi$).
Let $\bary_{\barx} := \bary[0,\infty)$.
The set $Y := \seb \bary_{\barx} \mid \barx \in \barX_f \sen$
 is a subspace of $\prod_{n \bpi}{\sC_n}^{\Z}$ with the product topology.
We denote the map $\phi : \barX_f \to Y$ by $\phi(\barx) = \bary_{\barx}$.
Evidently, $\phi$ is a bijection.
Obviously, $\phi$ is continuous and a homeomorphism.
Let $\sigma : Y \to Y$ be the left shift.
Then, it is easy to see that $\sigma \circ \phi = \phi \circ \sigma$.
Therefore, $(Y,\sigma)$ is topologically conjugate to $(\barX_f, \sigma)$.
It is easy to check that after the slides, the cuts do not have conflicts
 in different levels, i.e.,
 for $n > m \ge 0$, if an $n$-cut occurred at position $i$,
 then an $m$-cut has to occur at position $i$ (see \cref{fig:slide}).
For each $n \bpi$, originally,
 each $c_{n,i}$ is projected to $c_{n-1,1} c_{n-1,a(n)} d_{n-1,i}$
 by the graph map $\fai_{n}$.
After the slides, $c_{n,i}$ in the $n$th level is projected to
 $c_{n-1,a(n)} d_{n-1,i} c_{n-1,1}$ (see \cref{fig:slide}).
We write $\fai'_n(c_{n,i}) := c_{n-1,a(n)} d_{n-1,i} c_{n-1,1}$
 for each $1 \le i \le r_n$.
We now construct an ordered Bratteli diagram.
Let $V_0 := \seb v_0 \sen$ as usual.
For each $n \bpi$, let $V_n := \sC_n$.
We can write
 $\fai'_n(c_{n,i}) = c_{n-1,a(n)} c_{n-1,a(n,i,3)} c_{n-1,a(n,i,4)} \dotsb c_{n-1,a(n,i,k(n,i))} c_{n-1,1}$.
We make the minimal edge $e_{n,1}$ from $c_{n-1,a(n)}$ to $c_{n,i}$
 and all the rest in this order.
Thus, the minimal edge connects $c_{n-1,a(n)}$ to $c_{n,i}$
 regardless of $i$,
 and the maximal edge connects $c_{n-1,1}$ to $c_{n,i}$ regardless of $i$.
The ordered Bratteli diagram thus constructed is denoted as $(V',E',\ge')$.
Evidently, it has rank $K$.
It is easy to check that $(V',E',\ge')$ is properly ordered.
The simplicity follows from the simplicity of $\Gcal$.
The Bratteli--Vershik system thus constructed is identical to the symbolic
 system $(Y,\sigma)$.
This completes the proof.
%
%
\begin{figure}
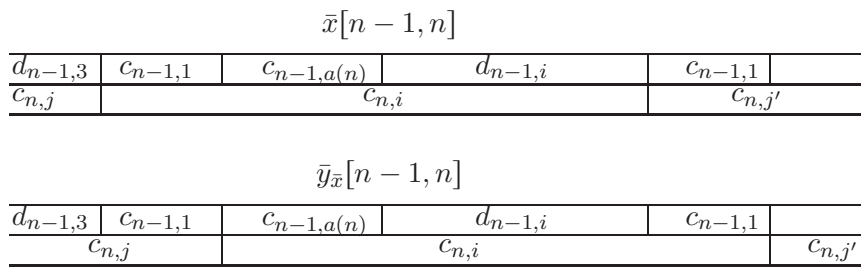

\begin{center}\leavevmode 
\xy
(50,42)*{\barx[n-1,n]},
(0,38)*{}; (112,38)*{} **@{-},
 (46,36)*{
 \hspace{13mm} d_{n-1,3}
 \hspace{4mm} c_{n-1,1} \hspace{8mm}
 \hspace{1mm} c_{n-1,a(n)} \hspace{12mm}
 \hspace{2mm} d_{n-1,i}  \hspace{16mm}
 \hspace{2mm} c_{n-1,1} \hspace{5mm}
 \hspace{1mm} },
(12,38)*{}; (12,34)*{} **@{-},
(28,38)*{}; (28,34)*{} **@{-},
(49,38)*{}; (49,34)*{} **@{-},
(84,38)*{}; (84,34)*{} **@{-},
(100,38)*{}; (100,34)*{} **@{-},
(0,34)*{}; (112,34)*{} **@{-},
(56,32)*{
 \hspace{2mm} c_{n,j} \hspace{7mm}\
 \hspace{32mm} c_{n,i} \hspace{35mm}
 \hspace{8mm} c_{n,j'}
 \hspace{12mm} },
(12,34)*{}; (12,30)*{} **@{-},
(84,34)*{}; (84,30)*{} **@{-},
(0,30)*{}; (112,30)*{} **@{-},
%
%
(50,22)*{\bary_{\barx}[n-1,n]},
(0,18)*{}; (112,18)*{} **@{-},
 (46,16)*{
 \hspace{13mm} d_{n-1,3}
 \hspace{4mm} c_{n-1,1} \hspace{8mm}
 \hspace{1mm} c_{n-1,a(n)} \hspace{12mm}
 \hspace{2mm} d_{n-1,i}  \hspace{16mm}
 \hspace{2mm} c_{n-1,1} \hspace{5mm}
 \hspace{1mm} },
(12,18)*{}; (12,14)*{} **@{-},
(28,18)*{}; (28,14)*{} **@{-},
(49,18)*{}; (49,14)*{} **@{-},
(84,18)*{}; (84,14)*{} **@{-},
(100,18)*{}; (100,14)*{} **@{-},
(0,14)*{}; (112,14)*{} **@{-},
(56,12)*{
 \hspace{12mm} c_{n,j} \hspace{7mm}\
 \hspace{32mm} c_{n,i} \hspace{35mm}
 \hspace{8mm} c_{n,j'}
 \hspace{2mm} },
(28,14)*{}; (28,10)*{} **@{-},
(100,14)*{}; (100,10)*{} **@{-},
(0,10)*{}; (112,10)*{} **@{-},
\endxy
\end{center}
\caption{In $[n-1,n]$ lines, $\barx$ is slid to $\bary_{\barx}$.}
\label{fig:slide}
\end{figure}
\end{proof}
Let $(X,f)$ be a Cantor minimal continuous surjection
 with topological rank $K > 1$.
By our result, it is not possible to conclude that
 the natural extension has topological rank $L > 1$.
Nevertheless, by \cite[Theorem 6.1]{SHIMOMURA_2015nonhomeomorphic},
 we can conclude that the natural extension is expansive.
Therefore, the natural extension is not an odometer.
Thus, by \cite[Theorem 6.1]{SHIMOMURA_2015nonhomeomorphic},
 we can conclude that the natural extension has topological rank $L > 1$.
Next, let $(\Sigma,\sigma)$ be a two-sided minimal subshift with
 finite topological rank
 $K > 1$.
By the \textit{one-sided factor}, we mean the one-sided minimal subshift
 $(\Sigma^+,\sigma)$ that is made by cutting off negative coordinates.
Then, the natural extension of $(\Sigma^+,\sigma)$ is canonically isomorphic
 to $(\Sigma,\sigma)$.
Thus, if $L$ is the topological rank of $(\Sigma^+,\sigma)$,
 then our main result concludes $L \ge K$.
By \cite[Theorem 6.8]{Shimomura_2016ZeroDimAlmostExtOdmeterGraphCov},
 we can find a two-sided subshift $(\Sigma',\sigma)$ that is topologically 
 conjugate to $(\Sigma,\sigma)$ such that the one-sided factor
 $({\Sigma'}^+,\sigma)$ has topological rank $K$.
Thus, we get the next corollary:
\begin{cor}
Let $(\Sigma,\sigma)$
 be a two-sided minimal subshift with topological rank $K$.
Then, the one-sided factor has topological rank $ \ge K$.
Furthermore,
 there exists a two-sided minimal subshift $(\Sigma',\sigma)$ such that
 $(\Sigma',\sigma)$ is topologically conjugate to $(\Sigma,\sigma)$ and
 the one-sided factor has topological rank $K$.
\end{cor}

\vspace{2mm}

\noindent
\textsc{Acknowledgments:}
This work was partially supported by JSPS KAKENHI (Grant Number 16K05185).
%

\end{document}